%% file: Fractional_operators_Dirichlet_averages_and_splines.tex
\documentclass[graybox]{svmult}

\usepackage{amsfonts}
\usepackage{amssymb}
\usepackage{amscd}
\usepackage{amsbsy}
\usepackage{graphicx}
\usepackage{mathrsfs} 

\usepackage[OT2,OT1]{fontenc}  
\newcommand\cyr{%
\renewcommand\rmdefault{wncyr}%
\renewcommand\sfdefault{wncyss}%
\renewcommand\encodingdefault{OT2}%
\normalfont \selectfont} \DeclareTextFontCommand{\textcyr}{\cyr}

\usepackage{mathptmx}       
\usepackage{helvet}         
\usepackage{courier}        
\usepackage{type1cm}        
\usepackage{makeidx}         
\usepackage{graphicx}        
\usepackage{multicol}        
\usepackage[bottom]{footmisc}

\def\N{{\mathbb N}} 
\def\Z{{\mathbb Z}}

\def\R{{\mathbb R}}
\def\C{{\mathbb C}}

\def\gdw{\iff}
\def\Norm{\Vert}

\def\Im{\mathrm{Im}\,}
\def\Re{\mathrm{Re}\,}
\def\supp{\mathrm{supp}\,}

\def\skpl{\langle}      
\def\skpr{\rangle}      

\def\const{\mbox{const.}}

\newcommand{\mcS}{\mathcal{S}}
\newcommand{\mcP}{\mathcal{P}}

\newcommand{\mcF}{\mathcal{F}}
\newcommand{\mcD}{\mathcal{D}}
\newcommand{\mcL}{\mathcal{L}}
\newcommand{\st}{\,:\,}
\newcommand{\be}{\begin{equation}}
\newcommand{\ee}{\end{equation}}
\newcommand{\inn}[2]{{\langle #1,#2 \rangle}}
\newcommand{\bB}{\mathcal{B}} 
\newcommand{\diffm}{\nabla _{\!\!\!-}}
\newcommand{\diffp}{\nabla _{\!\!+}}
\newtheorem{remarks}[theorem]{Remarks}

\makeindex             

\begin{document}

\title*{Fractional Operators, Dirichlet Averages and Splines}
\author{Peter Massopust}
\institute{Peter Massopust
\at Technische Universit\"at M\"unchen\\ Zentrum Mathematik\\ Lehrstuhl Mathematical Modelling (M6)\\Boltzmannstr. 3\\ 85747 Garching b. M\"unchen, Germany\\
\email{massopust@ma.tum.de}
\and
and\at Helmholtz Zentrum M\"unchen\\Ingolst\"adter Landstr. 1\\85764 Neuherberg, Germany}
%
%
\maketitle
\textbf{\em Dedicated to Paul Leo Butzer on the occasion of his 85th birthday}
\vspace*{48pt}
\abstract*{Fractional differential and integral operators, Dirichlet averages, and splines of complex order are three seemingly distinct mathematical subject areas addressing different questions and employing different methodologies. It is the purpose of this paper to show that there are deep and interesting relationships between these three areas. First a brief introduction to fractional differential and integral operators defined on Lizorkin spaces is presented and some of their main properties exhibited. This particular approach has the advantage that several definitions of fractional derivatives and integrals coincide. We then introduce Dirichlet averages and extend their definition to an infinite-dimensional setting that is needed to exhibit the relationships to splines of complex order. Finally, we focus on splines of complex order and, in particular, on cardinal B-splines of complex order. The fundamental connections to fractional derivatives and integrals as well as Dirichlet averages are presented.}

\abstract{Fractional differential and integral operators, Dirichlet averages, and splines of complex order are three seemingly distinct mathematical subject areas addressing different questions and employing different methodologies. It is the purpose of this paper to show that there are deep and interesting relationships between these three areas. First a brief introduction to fractional differential and integral operators defined on Lizorkin spaces is presented and some of their main properties exhibited. This particular approach has the advantage that several definitions of fractional derivatives and integrals coincide. We then introduce Dirichlet averages and extend their definition to an infinite-dimensional setting that is needed to exhibit the relationships to splines of complex order. Finally, we focus on splines of complex order and, in particular, on cardinal B-splines of complex order. The fundamental connections to fractional derivatives and integrals as well as Dirichlet averages are presented.
}
\section{Introduction}\label{sec1}
The theory of fractional differential and integral operators is currently a very active area of research extending results that hold for the traditional integer order operators and equations to fractional and even complex orders. There are several different types of such fractional differential and integral operators and only the specific application singles out which type is the most useful for the setting at hand. However, there exist certain types of function spaces on which some of the different notions of fractional derivative and integral operator coincide. This is the approach taken here. We consider Lizorkin spaces as the domains for fractional derivative and integral operators and show that they form endomorphisms and satisfy all those properties that one wishes to adapt from the traditional integral order scenario. This approach will set the stage for the main purpose of this paper, namely the exhibition of deep and interesting connections to two other seemingly very different mathematical subject areas: Dirichlet averages and splines.

The second topic we present are Dirichlet averages or Dirichlet means. Dirichlet averages were introduced by Carlson in 1977 with the main objective to construct functions that follow "the principle that the fundamental functions should be free of unnecessary branch points and unnecessary transformations." (\cite[p. ix]{carlson}) This objective then led to the definition of the $R$-- and $S$--hypergeometric functions that generalize in a very natural way the traditional monomials and exponential functions. Later,  some relations to the classical theory of splines were made and it was found that B-splines can be represented by Dirichlet averages over $\delta$-distributions \cite{carlson91}. Another application of Dirichlet averages was presented in \cite{castell} where Dirichlet splines were shown to be fractional integrals of B-splines. Here a first indication surfaced that there might be a connection between fractional operators, Dirichlet averages and splines. 

Recently, a new class of splines was investigated in a series of publications \cite{forster06,FM07,FM08,FM09,FM09a,FM11,FM11a,FM11b,FMU,M09,M12,M12a,FM10}. These so-called complex B-splines extend the classical notion of cardinal polynomial B-spline to complex order. They are the third topic we consider. It was found that both fractional operators and Dirichlet averages are needed to understand the analytic and geometric structure of these splines. It is interesting to note that the analog of complex B-splines had already been introduced in an earlier paper by Zheludev \cite{zhel} in connection with fractional derivatives and convolution equations but their properties were not investigated further. Here, we define splines of complex order as the distributional solutions of a certain fractional differential equation and show that the complex B-spline is a solution. We then use this set-up to exhibit the specific relationships between the three topics mentioned above.

The structure of the paper is as follows. After introducing some preliminaries and setting notation, we present fractional differential and integral operators in Section \ref{sec3}. In Section \ref{sec4}, Dirichlet averages are introduced and their main properties discussed. Splines of complex order are defined in Section \ref{sec5} where we relate all three concepts.
\section{Notation and Preliminaries}\label{sec2}
In the sequel, we make use of multi-index terminology and employ the following notation. $\N$ denotes the set of positive integers and $\N_0 := \N\cup \{0\}$ the set of nonnegative integers. The ring of integers is denoted by $\Z$ and $\Z_0^\pm := \{a\in \Z\st \pm a \geq 0\}$. The following subsets of the real numbers $\R$, respectively, complex numbers $\C$ are frequently encountered: $\R^+_0 := \{r\in \R\st r \geq 0\}$, $\R^\pm := \{r\in \R\st \pm r > 0\}$, and $\C^+ := \{z\in \C\st \Re z > 0\}$. 

For multi-indices $m:= (m_1,\ldots, m_n)$ and $r := (r_1, \ldots, r_n)$, and for $p\in\R$, define $m^p$ and $m^r$ to be
\[
m^p := (m_1^{p}, \ldots, m_n^{p}) \quad\textrm{and}\quad m^r := (m_1^{r_1},\ldots, m_n^{r_n}),
\]
respectively.  The length of a multi-index is defined as $|m| := m_1 + \cdots m_n$. 

For $\alpha\in\R$ and multi-indices $k$ and $m$, we also define $k\prec \alpha$ iff $k_i\prec\alpha$ and $k \prec m$ iff $k_i \prec m_i$, $\forall i \in \{1, \ldots, n\}$, where $\prec$ is any one of the binary relation symbols $<, \leq, >, \geq, =$. For $\alpha\in\C$ and a multi-index $k$, we set $k \pm \alpha := (k_1 \pm \alpha, \ldots, k_n \pm \alpha)$ and $\alpha k := (\alpha\,k_1,\ldots, \alpha\,k_n)$.

For $a,b\in \R^n$ we write
\[
\int_a^b f(t)\, dt := \int_{a_1}^{b_1}\cdots\int_{a_n}^{b_n} f(t_1, \ldots, t_n)\, dt_1\ldots dt_n,
\]
where $-\infty\leq a < b\leq +\infty$.

Finally, for $z = (z_1, \ldots, z_n)\in (\C\setminus \Z_0^-)^n$, we set
\[
\Gamma (z) := \Gamma (z_1) \cdots\Gamma (z_n),
\]
where $\Gamma: \C\setminus \Z_0^-\to \C$ denotes the Euler gamma function. 

The Schwartz space $\mcS (\R^n, \C)$, $n\in \N$, is defined by
\[
\mcS (\R^n, \C) := \left\{ \varphi \in C^\infty(\R^n, \C) \,\Big|\, \forall k,m \in \mathbb{N}_0^n: \; \sup_{x\in\R^n} \|x^k D^m \varphi(x) \|  < \infty\; \right\}
\]
Here, $D^m := \displaystyle{\frac{\partial^{m_1}}{\partial x_1^{m_1}} \cdots \frac{\partial^{m_n}}{\partial x_n^{m_n}}} = \displaystyle{\frac{\partial^{m_1+\cdots + m_n}}{\partial x_1^{m_1}\cdots \partial x_n^{m_n}}}: C^\infty (\R^n, \C)\to C^\infty (\R^n, \C)$ denotes the ordinary partial derivative operator. The Schwartz space $\mcS(\R^n, \C)$ is a metrizable locally convex topological vector space whose topology is induced by the semi-norms
\[
\|f\|_M = \sup_{x \in \R^n} \max_{|k|,\, |m| < M} \|x^k D^m f(x)\|.
\]
For $1\leq p \leq \infty$, the following embeddings hold:
\[
\mcS (\R^n, \C) \subset L^p (\R^n,\C).
\]
Moreover, for $1\leq p < \infty$, the Schwartz space $ \mcS (\R^n, \C)$ is dense in $L^p(\R^n,\C)$ with respect to the $L^p$--norm. The space $\mcS (\R^n, \C)$ is closed under pointwise multiplication, i.e., if $f, g \in \mcS (\R^n, \C)$ then their product $fg\in \mcS (\R^n, \C)$. This is a consequence of the Leibniz rule for derivatives. 
\section{Fractional Operators}\label{sec3}
Fractional differential and integral operators have a long history going back to Newton and Leibnitz. These operators extend the classical concepts of differentiation and integration that involve integral orders to fractional orders. Later, more general operators of fractional order were considered. (See, for instance, \cite{butzerwestphal,hille,westphal} and references therein.) Currently, fractional derivative and integral operators are an active area of research extending from the theoretical foundations to applied sciences, such as physics and engineering. (A very short and incomplete list of recently published books is \cite{baleanu,hilfer,ortigueira}. The reader may find more references there.)

For the purposes of this paper, we introduce a type of fractional differential and integral operator that includes such well-known operators as Riemann-Liouville and Caputo by defining them on a particular space of functions. 


For our later purposes, we introduce the following subspace of the Schwartz space $\mcS := \mcS(\R^n, \C)$. This and related spaces play an important role in the theory of fractional differentiation and integration as they allow the identification of several fractional differential and integral operators \cite{samko}. 

The space $\Psi := \Psi (\R^n, \C)\subset \mcS$ is defined by
\begin{eqnarray*}
\Psi := \{\psi\in \mcS\st (D^m \psi)(x_1, \ldots, x_{k-1}, 0, x_k, \ldots, x_n) = 0, \,\forall\,m \in \N_0^n,\,k\in \{1, \ldots, n\} \}.
\end{eqnarray*}
An example of a function $\psi\in \Psi$ is given by
\[
\psi (x) = \textrm{exp} \left(-\|x\|^2 - \sum_{k=1}^n x_k^{-2}\right).
\]
The space $\Psi$ can be made into a topological vector space by means of a countable number of semi-norms $\{p_{k,m,p}\st k,m,p\in\N_0^n\}$ of the form
\[
p_{k,m,p} (x) := \sup\left\{ (1+x^2)^{m/2}\,\|x\|^{-p}\,\|(D^k \psi)(x)\|\st x\in \R^n\right\}.
\]
Considering the cases $\|x\| <1$ and $\|x\|\geq 1$, it follows that this topology coincides with the topology on $\mcS$. (See, for instance, \cite{samko}.)

The {\em Lizorkin space} $\Phi := \Phi (\R^n, \C)$ (cf. \cite{lizorkin}) is given by
\[
\Phi := \{\varphi\in \mcS\st \mcF(\varphi) \in \Psi\}.
\]
Here, we defined the Fourier--Plancherel transform $\mcF: \mcS\to \mcS$ by
\[
\mcF (f) (\omega) = :\widehat{f} (\omega):= \int_{\R^n} f(x)\, e^{i \inn{x}{\omega}} dx,
\]
where $\inn{\bullet}{\bullet}:\R^n\times\R^n\to \C$ denotes the canonical inner product in $\R^n$. The space $\Phi$ may be considered as a topological vector space when endowed with the topology of $\mcS$. Note that since $\Phi$ and $\Psi$ are isomorphic under the Fourier transform, i.e., $\Psi = \mcF\Phi$, the Lizorkin space $\Phi$ has a rich supply of functions. 

The next results summarize some properties of the Lizorkin space $\Phi$. For proofs and more details, we refer the interested reader to \cite{samko} or \cite{troyanov}.

\begin{proposition} \label{prop1}
\begin{enumerate}
\item The Lizorkin space $\Phi$ consists of those and only those functions $\varphi\in \mcS$, all of whose moments vanish along all coordinate axes. More precisely,
\[
\int_\R x^m\,\varphi (t_1, \ldots, t_{k-1}, x, t_{k+1}, \ldots, t_n) \, dx = 0, \;\forall k \in \{1, \ldots, n\},\,\forall m\in \N_0.
\]
In other words., if $P_k \varphi$ denotes the projection of $\varphi$ onto $\Phi(\boldsymbol{0}^{k-1}\times\R\times \boldsymbol{0}^{n-k})$, $k \in \{1, \ldots, n\}$, then 
\[
P_k \varphi\, \bot\, (\bullet)^m, \quad \forall m\in \N_0,
\]
where $\bot$ denotes orthogonality with respect to the $L^2$-inner product on $\R$.

\item $\Phi$ is a closed subset of $\mcS$ and an ideal under convolution, i.e., if $\varphi\in \Phi$ and $f\in \mcS$, then $\varphi*f \in \Phi$.

\item Let $\mcS^\prime := \mcS^\prime (\R^n, \C)$ be the topological dual of $\mcS$. Then the topological dual $\Phi^\prime$ of $\Phi$ is given by
\[
\Phi^\prime = \mcS^\prime/\mcP,
\]
where, $\mcP := \{f\in \mcS^\prime\st \supp \mcF (f) = \{0\}\}$ denotes the set of polynomials in $\mcS^\prime$.
\end{enumerate}
\end{proposition}

\begin{proposition}\label{prop2}
Given $\psi\in \mcS$, the following statements are equivalent.
\begin{enumerate}
\item	$\psi\in \Psi$;
\item	$\forall\mu\in \N^n$, $\forall t\in \R^+:$ $D^\mu\psi(\xi) \in o(\|\xi\|^t)$ as $\|\xi\|\to 0$;
\item	$\forall m\in \N^n:$ $\|\xi\|^{-2m}\,\psi \in \mcS$.
\end{enumerate}
\end{proposition}
Regarding the spaces $\mcP$ and $\Phi$, we remark that:
\begin{enumerate}
\item $f\in \mcP$ if and only if $\Delta^m f = 0$, for some $m\in \N^n$, where $\Delta:\mcS\to \mcS$, 
$$
\Delta := -\sum_{j=1}^n \frac{\partial^2}{\partial x_j^2},
$$
denotes the Laplace operator on $\mcS$. Using the Laplace operator, the Lizorkin space $\Phi$ may also be defined as (cf. \cite{troyanov})
\[
\Phi := \bigcap_{m\in \N} \Delta^m (\mcS), 
\]
\item $\Phi$ is dense in $L^p (\R^n)$ for $2\leq p < \infty$. \cite{samko2}. 
\end{enumerate}
\noindent
For $t\in \R^n$ and $z\in \C^n$, we define
\[
t^z_+ := 
\left\{ \begin{array}{ cl}
t^z = e^{z\log t} = t^{\Re z}\,e^{i \Im z \log t},      &\qquad t \geq 0;   \\
    0,  &   \qquad t < 0,
\end{array}
\right.
\]
where $t^z = (t_1^{z_1}, \ldots, t_n^{z_n}) = (e^{z_1\log t_1}, \ldots, e^{z_n\log t_n})$.
\begin{definition}
Let $z\in (\C^+)^n$ be a multi-index. The fractional integral $\mcD^{-z} :\Phi \to \Phi$ is defined by
\[
(\mcD^{-z} \varphi)(x) := \frac{1}{\Gamma (z)}\,\int_{\R^n} t^{z - 1}_+ \varphi(x + t) dt.
\]
\end{definition}
\begin{definition}
Let $n\in \N$. For all $i = 1, \ldots, n$, let $m_i := \lceil \Re z_i \rceil$, where $\lceil\bullet\rceil : \R\to\Z$, $r \mapsto \min\{k\in \Z\st r\leq k\}$,  denotes the ceiling function. Set $\nu_i := m_i - z_i$ and $m := \sum m_i$. The fractional derivative $\mcD^{z} :\Phi \to \Phi$ is given by
\begin{eqnarray*}
(\mcD^{z} \varphi)(x) &:= &\displaystyle{\frac{1}{\Gamma (\nu)}\,\frac{\partial^m}{\partial x_1^{m_1}\cdots \partial x_n^{m_n}}\,\int_{\R^n} t_+^{\nu -1} \varphi(x+ t) dt},\\
& = & \displaystyle{\frac{1}{\Gamma (\nu)}\,\int_{\R^n} t_+^{\nu -1} (D^m \varphi)(x + t) dt = D^m \mcD^{-\nu} \varphi.}
\end{eqnarray*} 
\end{definition}
\noindent
For $z\in \N^n$, these definitions reduce via the Cauchy formula to the classical derivative and integral.

\begin{remark}\label{Rem2.3}
It can be shown (cf. \cite{samko}) that $\mcD^{\pm z}$ leaves the Lizorkin space $\Phi$ invariant. Hence, the dual space $\Phi'$ is invariant under $\mcD^{\pm z}$. Since $\Phi\subset L^1 (\R^n, \C)$, we have the canonical inclusion $\Phi \subset \Phi'$ as topological vector spaces.
\end{remark}
\begin{remark}
\item The interchange of derivative and integral in the definition of $\mcD_\pm^z$ is justified since $f\in \Phi\subset \mcS$, hence vanishes at $\pm\infty$.
\end{remark}
\begin{remark}
\item Introducing a kernel function $K_z:\R^n\to\C$ by
\[
K_z(x) := \frac{x_+^{z-1}}{\Gamma (z)},
\]
and using Remark 2, we can write the definitions of the fractional integral and derivative in the following more succinct form:
\[
\mcD^z f := (D^m f) * K_{m-z} = D^m (f*K_{m-z}), \qquad m = \lceil \Re z \rceil,
\]
and
\[
\mcD^{-z} f := f*K_z.
\]
\end{remark}

We note that if function spaces other than $\Psi$ are considered, 
\be\label{notequal}
(D^m f) * K_{m-z} \neq D^m (f*K_{m-z}),
\ee
in general, and each side defines a different fractional derivative operator. The left-hand side of (\ref{notequal}) is called the {\em Caputo fractional derivative} and the right-hand side the {\em Riemann-Liouville} or {\em Weyl fractional derivative}. The interested reader may consult \cite{kilbas} or \cite{podlubny} for more details.

Next, we summarize some properties of the fractional derivatives and integrals introduced above.

\begin{theorem}\label{thm1}
The fractional derivatives $\mcD^z$ and fractional integrals $\mcD^{-z}$ have the following properties.
\begin{enumerate}
\item Semi-group property: For all $z_1, z_2\in (\C^+)^n$ and all $\varphi\in \Phi$ one has: 
\be\label{important}
\mcD^{\pm (z_1 + z_2)} \varphi =(\mcD^{\pm z_1}\varphi) (\mcD^{\pm z_2}\varphi) = (\mcD^{\pm z_2}\varphi) (\mcD^{\pm z_1}\varphi).
\ee
\item For all $z\in (\C^+)^n$ and all $\varphi\in \Phi$ one has: 
\[
\mcD^z \mcD^{-z} \varphi = \mcD^{-z} \mcD^{z} \varphi = \varphi.
\]
\end{enumerate}
\end{theorem}
\begin{proof}
For proofs of these properties, we refer the reader to \cite{anastassiou,podlubny,samko}. The generalization of some of these results to $\R^n$ lies at hand.\qed
\end{proof}

In the sequel, we also consider generalized functions $f\in \Phi'$ and $f\in \Psi'$ and, therefore, need to extend the concept of fractional derivative and integral to this more general setting.

To this end, let $x \in \R^n$ and $z\in (\C^+)^n$ with $\Re z > -1$. If we define
\[
x^z := \left\{ \begin{array}{ cl}
    x^z,      &\qquad x \geq 0;   \\
    e^{i\pi z} |x|^z, &\qquad x < 0,
\end{array}
\right.
\]
then we may regard the locally integrable function $x\mapsto x^z$, $\Re z > -1$, as an element of $\Phi^\prime$ by setting
\[
\skpl (\bullet)^z_, \varphi\skpr = \int_\R^n t^z \varphi (t) dt, \quad \forall\varphi\in \Phi.
\]
In passing, we note that  for all $\varphi\in \Psi$ the function $z\mapsto \skpl K_z, \varphi\skpr$ is holomorphic on $(\C\setminus \N_0)^n$.

\begin{remark}
The function $(\bullet)^{z}$ may also be defined for general $z\in \C$ via Hadamard's partie finie and represents then a pseudo-function. For more details, we refer the interested reader to \cite{dantray,gelfand}, or \cite{zemanian}.
\end{remark}

Employing the results in \cite{samko}, we mention that for $f,g\in \Psi^\prime$ the convolution exists on $\Psi^\prime$ and is defined in the usual way by 
\be\label{eq7}
\skpl f*g, \varphi\skpr :=  \skpl (f\times g) (x,y), \varphi (x+y)\skpr = \skpl f(x), \skpl g(y), \varphi (x+y)\skpr\skpr, \quad \varphi\in \Psi.
\ee
The pair $(\Psi^\prime, *)$ is a convolution algebra with the Dirac delta distribution $\delta$ as its unit element.

\begin{definition}
Let  $z\in \C^+$, let $\varphi\in \Psi$ be a test function and $f\in \Psi^\prime$. Then the fractional derivative operator $\mcD^{z}$ on $\Psi^\prime$ is defined by
\[
\skpl\mcD^{z} f, \varphi\skpr := \skpl (D^m f)*K_{m-z},\varphi\skpr,\quad m = \lceil\Re z\rceil,
\]
and the fractional integral operator $\mcD^{-z}$ by
\[
\skpl\mcD^{-z} f, \varphi\skpr := \skpl f*K_z,\varphi\skpr.
\]
\end{definition}

The semi-group properties (\ref{important}) also hold for $f\in \Psi^\prime$. (For a direct proof, see \cite{gelfand} or \cite{podlubny}.)

For later purposes, we need the restriction of $\Psi$ to $[0,\infty)^n$:
\[
\Psi_+ := \{f\in \Psi \st \supp f \subseteq [0,\infty)^n\}.
\]
Note that all the properties and results mentioned above also apply to the elements of $\Psi_+$ and $\Psi_+^\prime$.

\begin{example}
Suppose that $n := 1$. Let $z\in \C^+$. Then it can be shown (cf. \cite{gelfand,podlubny, samko}) that the $z$-th derivative of a shifted truncated power function is given 
\be\label{dertruncpow}
\mcD^z \left[\frac{(x - k)_+^{z-1}}{\Gamma (z)}\right] = \delta (x - k), \quad \R\ni k < x \in [0, \infty).
\ee
Thus, employing the semi-group properties (\ref{important}) of $\mcD^z$ or by direct computation using definition (\ref{eq7}) of convolution, one obtains
\be\label{intdelta}
\mcD^{-z} \delta (\bullet - k) = \frac{(\bullet - k)^{z-1}_+}{\Gamma (z)}, \qquad k\in \R.
\ee
\end{example}
Equations (\ref{dertruncpow}) and (\ref{intdelta}) will be reconsidered in Section 4 where they relate to splines to complex order.
\section{Dirichlet Averages}\label{sec4}
The concept of Dirichlet average or Dirichlet mean over a finite-dimensional simplex $\triangle^n$, $n\in \N$, was first discussed in \cite{carlson} and related to univariate and multivariate B-splines in \cite{carlson91}. Dirichlet averages generalize among others, classical functions such as $x^k$, $k\in \N_0$, $x^t$, $t > 0$, and $e^x$, and have produced deep and interesting connections to special functions.

In this section, we extend the notion of Dirichlet mean to an infinite-dimensional simplex $\triangle^\infty$ and show that under mild conditions the results important for our interests do also hold on $\triangle^\infty$. In particular, we show that using a geometric interpretation, the type of fractional derivative and integral introduced in the previous section can be applied to Dirichlet averages.

To this end, let $n\in \N$ and denote by $\triangle^n$ the standard $n$-simplex in $\R^{n+1}$:
\[
\triangle^n := \Bigg\lbrace u:=  (u_0,\ldots,u_n)\in\mathbb{R}^{n+1} \,\Bigg\vert\;  u_j \geq 0;\; j = 0,1,\ldots, n;\, \sum_{j=0}^n u_j = 1\Bigg\rbrace.
\]

The extension of $\triangle^n$ to infinite dimensions is done by using projective limits. The resulting infinite-dimensional standard simplex is given by
\[
\triangle^\infty := \lim_{\longleftarrow} \triangle^n = \left\{ u := (u_j)_{j\in (\R_0^+)^{\N_0}}\,\Bigg\vert\; \sum_{j=0}^\infty\, u_j =1\right\},
\]
and endowed with the topology of pointwise convergence, i.e., the weak$*$-topology. We denote by $\mu_b = \displaystyle{\lim_{\longleftarrow}} \mu_b^n$ the projective limit of the \textit{Dirichlet measures $\mu_b^n$} defined on the $n$-dimensional standard simplex $\triangle^n$ through the density functions
\be\label{dirmeas}
\frac{\Gamma(b_0)\cdots \Gamma(b_n)}{\Gamma(b_0 + \cdots + b_n)}\,u_0^{b_0-1}\,u_1^{b_1-1}\cdots u_n^{b_n-1},
\ee
where $b :=(b_0, \ldots, b_n) \in \C^{n+1}$ with $\Re b_j > 0$, $j = 0, 1, \ldots, n$. Note that by the Kolmogorov Extension Theorem this measure $\mu_b$ exists. (See, for instance, \cite{shiryayev}.)

\begin{definition}
Let $\Omega$ be a nonempty convex open subset of $\mathbb{C}^s$, $s\in \N$, and let $\tau\in\Omega^n$. The Dirichlet average of a measurable function $f:\Omega\to\mathbb{C}$ is defined as the integral
\be\label{dirichletaverage}
F(b;\tau) := \int_{\triangle^n} f(\tau\cdot u) d\mu^n_b (u),
\ee
where $u\cdot \tau := \displaystyle{\sum_{i=0}^n}\, u_i \tau^i \in \C^s$.
\end{definition}

We note that it is customary to denote the Dirichlet average of a function $f$ by the corresponding upper-case letter, $F$. It can be shown that the Dirichlet average of a derivative equals the derivative of the Dirichlet average. For more details regarding the properties of Dirichlet averages and their connection to the theory of special functions, we refer the interested reader to \cite{carlson}, where one may also find the proof of the next result.

\begin{proposition}
Suppose that $f:\Omega \to \C$ is holomorph. Then, for fixed $\tau\in\Omega^{n+1}$, the Dirichlet average $F(\,\bullet\,, \tau)$ is a holomorphic function on $(\C^+)^{n+1}$.
\end{proposition}

The extension of (\ref{dirichletaverage}) to an infinite-dimension setting proceeds as follows. Let $\Omega$ to  be a nonempty open convex set in $\C^s$, $b\in (\C^+)^{\N_0}$ and $f\in \mcS(\Omega):= \mcS(\Omega,\C)$ a measurable function. For $\tau\in \Omega^{\N_0}\subset (\C^s)^{\N_0}$ and $u\in \triangle^\infty$, define $u\cdot \tau$ to be the bilinear mapping $(u,\tau)\mapsto \displaystyle{\sum_{i=1}^\infty}\, u_i \tau^i$. The infinite sum exists whenever 
\be\label{cond}
\limsup_{n\to\infty} \sqrt[n]{\|\tau^n\|} < \infty, 
\ee
where $\|\cdot\|$ denotes the canonical Euclidean norm on $\C^s$. (See also \cite{FM09}.) We refer to $\tau = (\tau^1, \ldots, \tau^n, \ldots)$ as a knot vector and to $b\in (\C^+)^{\N_0}$ as the weight vector associated with the Dirichlet average. 

\begin{definition} The Dirichlet average $F:(\C^+)^{\N_0}\times\Omega^{\N_0}\to \C$ on $\triangle^\infty$ is defined by
\[
F(b;\tau) := \int_{\triangle^\infty} f(u\cdot\tau)\,d\mu_b(u),
\]
where $\mu_b = \displaystyle{\lim_{\longleftarrow} }\;\mu_b^n$ is the projective limit the Dirichlet measures $\mu_b^n$ whose density functions are given by (\ref{dirmeas}).
\end{definition}

Under the assumption that $f:\Omega\to\C$ is a holomorphic function and the knot vector $\tau$ satisfies condition (\ref{cond}), the Dirichlet average on $\triangle^\infty$ exists and is holomorph on $(\C^+)^\infty$ for fixed $\tau\in\Omega^{\N_0}$. 

Using the fact that $\triangle^\infty$ is the projective limit of its finite-dimensional projections $\triangle^n$, $n\in\N$, the following known properties of $F$ extend naturally to the infinite-dimensional setting. (Cf. also \cite{FM10}.)
\begin{itemize}
\item Let $\sigma: \N_0^\infty\to\N_0^\infty$ be a permutation. Then 
	\[
	F(b_{\sigma(0)}, b_{\sigma(1)}, \ldots, ; \tau^{\sigma(0)}, \tau^{\sigma(1)}, \ldots) = F(b_0,b_1, \ldots; \tau^0, \tau^1, \ldots);
	\]
\item $F(b_0, b_1, b_2, \ldots ; \tau^1, \tau^1, \tau^2, \ldots) = F(b_0+b_1, b_2, \ldots ; \tau^1, \tau^2, \ldots)$;
\item $F(0, b_1, b_2, \ldots ; \tau^0, \tau^1, \tau^2, \ldots) = F(b_1, b_2, \ldots ; \tau^1, \tau^2, \ldots)$;
\item If $\tau = (z,z,z,\ldots)\in \Omega^{\N_0}$, then $u\cdot \tau = z\,\displaystyle{\sum_{i=0}^\infty}\, u_i = z\in \C^s$, and thus $F(b;\tau) = f(z)$.
\end{itemize}

Now suppose that the weight vector $b\in \ell^1 (\N_0,\R_+)$. Let $c:= \displaystyle{\sum_{i=0}^\infty}\; b_i$ and $w_i := \frac{b_i}{c}$. Since for $j = 1, \ldots, n$, the equation
\[
\int_{\triangle^n} u_j d\mu_b^n (u) = \frac{b_j}{\displaystyle{\sum_{i=0}^n b_i}}
\]
holds for all finite-dimensional projections $\triangle^n$ of $\triangle^\infty$ (cf. \cite{carlson}), we have
\[
\int_{\triangle^\infty} u_j d\mu_b (u) = \frac{b_j}{c} = w_j, \quad\forall j\in \N_0.
\]
In a similar fashion, using the fact that (Section 4.4 in \cite{carlson}), 
\[
u_1^{m_1}\cdots u_k^{m_k} d\mu_b^k (u) = \frac{(\Gamma(b_1 + \cdots b_k))(\Gamma(b_1 + m_1)\cdots \Gamma(b_k+m_k))}{(\Gamma(b_1)\cdots \Gamma(b_k))(\Gamma(b_1 + m_1 + \cdots + b_k+m_k))} \,d\mu_{b+m}^k (u),
\]
for $m\in \N^k$, one obtains the identity
\be\label{10}
u_j d\mu_b (u) = w_j d\mu_{b + e_j} (u), \quad j \in \N_0,
\ee
generalizing the corresponding finite-dimensional identity. (Cf. \cite{carlson}.) (Here $e_j := \{\delta_{i,j}\st i\in \N_0\}$.) From (\ref{10}) one obtains the identiy
\[
\int_{\triangle^\infty} f(u\cdot\tau) d\mu_b(u) = \sum_{j=0}^\infty w_j \int_{\triangle^\infty} f(u\cdot\tau) d\mu_{b+e_j}(u),
\]
or, equivalently,
\[
F(b;\tau) = \sum_{j=0}^\infty w_j F(b+e_j; \tau).
\]
In particular, for $x\in \R^s$ and $g(x) := x f(x)$, this last equation gives
\[
G(b;\tau) = \sum_{j=0}^\infty w_j \tau^j F(b+e_j; \tau),
\] 
where $\tau^j\in \C^s$ is the $j$th component of $\tau$.

The results regarding the relations between Dirichlet averages found in \cite{carlson91}, Section 5, or \cite{neuman}, Section 3, transfer easily to the infinite-dimensional setting using the definition of projective limit. In addition, several other identities can be derived using this extension and infinite-dimensional analogs of special functions, such as for instance the Lauricella $F_B$ function, defined. For more details, we refer to \cite{M09,M12,FM10}.

Of particular interest are fractional derivatives of Dirichlet averages and their relation to the Dirichlet averages of fractional derivatives. (See also \cite{FM10}, where these notions were considered in the context of Weyl fractional derivatives and integrals.) 

To this end, let $f\in \Psi_+$. Let $z\in \C^+$ and let $n:= \lceil\Re z\rceil$. Furthermore, let $x := (x_1, \ldots, x_s)^\top\in [0,\infty)^s$. The partial fractional derivative $\mcD^z_i$ with respect to $x_i $, $i = 1, \ldots, s$, of order $z$ is defined by
\[
\mcD_i^z f (x) := (D^n_i f)*K_{n-z}^i = \frac{1}{\Gamma (z)} \int_\R (D_i^n f)(\xi) (x_i - \xi)_+^{z-1} d\xi,
\]
where $D_i^n := \frac{\partial^n}{\partial x_i^n}$.

Consider for a moment the case $s:=1$. Let $\tau\in\Omega^k$ and denote by $\partial_i := \frac{\partial}{\partial\tau_i}$, $i = 0,1,\ldots, k$, the partial derivative operator. Let $f:\Omega\to \R$ and consider the operator $\displaystyle{\sum_{i=0}^k} \;\partial_i = \inn{\nabla}{e(k)}$, where $\nabla$ is the $k$-dimensional gradient and $e(k) = (1, \ldots, 1)\in\N^k$. However, $\inn{\nabla}{e(k)} f$ is equal to the \textit{univariate} derivative $\displaystyle{\frac{d}{dx}} f(x,\ldots, x)$, where $x:= \tau_1 = \cdots = \tau_k$. This suggests the following definition. 

\begin{definition} Let $g\in \Psi_+$ and let $z\in \C^+$ with $\lceil\Re z\rceil \leq k$. Then
\begin{eqnarray*}
\left(\sum_{i=0}^k \partial_i\right)^z g (\tau) & := \displaystyle{\frac{(-1)^n}{\Gamma (\nu)} \frac{d^n}{d x^n} \,\int_{\R_+} t^{\nu - 1} g(x+t, \ldots, x+t) dt}\\
& = \displaystyle{\frac{(-1)^n}{\Gamma (\nu)} \,\int_{\R_+} t^{\nu - 1} g^{(n)}(x+t, \ldots, x+t) \,dt}.
\end{eqnarray*}
\end{definition}
Employing this definition, it was shown in \cite{FM10} that
\[
(\mcD_i^z F) (b(k); \tau) = \int_{\triangle^k} u_i^z \,\mcD^{z}f (u\cdot \tau) d\mu_{b(k)}(u),
\]
and, if $\{i_1,\ldots, i_m\}\subseteq\{0,1, \ldots, k\}$, $m = 0,1, \ldots, k$,
\[
(\mcD_{i_1}^{z_{i_1}}\cdots \mcD_{i_m}^{z_{i_m}} F) (b(k); \tau) = \int_{\triangle^k} u_{i_1}^{z_{i_1}}\cdots u_{i_m}^{z_{i_m}} \,\mcD^{(z_{i_1}+\cdots+z_{i_m})}f (u\cdot \tau) d\mu_{b(k)}(u),
\]
where $z_{i_\ell}\in \C^+$, $\ell = 1, \ldots, m$,  and $\lceil \Re z_{i_1}\rceil + \cdots + \lceil \Re z_{i_m}\rceil \leq k$.

In order to extend the above results to $\triangle^\infty$, we need to consider the operator $\partial := \partial (\infty) := \displaystyle{\sum_{i=0}^\infty}\, \partial_i$, where $\partial_i$ denotes again the partial derivative with respect to $\tau_i$, $i\in\N_0$.
 
Let $f\in \Psi_+^\infty := \{f\in \Psi \st \supp f \subseteq \displaystyle{\prod_{n\in \N} [0,\infty)}\}$ and let $z\in \C^+$ with $n:= \lceil\Re z\rceil$ and $\nu:= n - z$. Define $\partial^z := \left(\displaystyle{\sum_{i=0}^\infty}\, \partial_i\right)^z$ to be the operator on $\Psi_+^\infty$ given by the expression
\[
(\partial^z f)(\tau) := \frac{(-1)^n}{\Gamma (\nu)}\,\frac{d^n}{dx^n}\,\int_0^\infty t^{\nu-1} f(x +t)dt = \frac{(-1)^n}{\Gamma (\nu)}\,\int_0^\infty t^{\nu-1} f^{(n)}(x+t)dt,
\]
where $\R\ni x:= \tau_1 = \tau_2 = \cdots = \tau_n = \cdots$. Replacing $f$ by the Dirichlet average $G(b;\bullet)$ of some function $g\in\Psi_+^\infty$ and for a weight vector $b\in(\C^+)^{\N_0}\cap\ell^1 (\N_0)$, one can show that the following identities hold:
\[
(\partial^z G(b;\bullet))(\tau) = \int_{\triangle^\infty} \mcD^{z}g (u\cdot \tau) d\mu_b (u),
\]
\[
(\mcD_i^{z_i} G(b;\bullet))(\tau) = \int_{\triangle^\infty} u_i^{z_i} \mcD^{z_i}g (u\cdot \tau) d\mu_b (u),
\]
and 
\be
(\mcD_{i_1}^{z_{i_1}}\cdots \mcD_{i_m}^{z_{i_m}} G (b;\bullet))(\tau) = \int_{\triangle^\infty} u_{i_1}^{z_{i_1}}\cdots u_{i_m}^{z_{i_m}} \mcD^{(z_{i_1}+\cdots+z_{i_m})}g (u\cdot \tau) d\mu_{b}(u),
\ee
for any $\{i_1,\ldots, i_m\}\subseteq\N_0$.

For more details and further discussions, we again refer the interested reader to \cite{M09,M12,FM10}.
\section{Polynomial Splines of Complex Order}\label{sec5}
In this section, we define a wide class of splines that encompasses the classical polynomial Schoenberg splines of integral order and the fractional and complex B-splines introduced in \cite{unserblu00} and \cite{forster06}. This definition is motivated by and based on the original ideas in \cite{karlin, unserblu05} to define splines via certain classes of differential operators and the description in \cite{masso}. We introduce these \textit{polynomial splines of complex order $z\in \C_1 := \{\zeta\in \C\st \Re \zeta > 1\}$} as follows.

\begin{definition}
Let $z\in \C_1$ and let $\{a_k\st k\in \N_0\}\in \ell^\infty (\R)$. A solution of the fractional differential equation
\be\label{cs}
\mcD^z f = \sum_{k=0}^\infty a_k\, \delta (\bullet - k)
\ee
is called a \textit{polynomial spline of complex order $z$}.
\end{definition}

\begin{remarks}
\begin{enumerate}
\item	If in (\ref{cs}) we set $z := n\in \N$, then we obtain the cardinal polynomial B-spline of order $n$ as a solution. (See, for instance, \cite{karlin, masso, unserblu05}.)
\item	The right-hand side of (\ref{cs}) may be interpreted as a weighted Dirac comb or \textcyr{Sh}--function.
\end{enumerate}
\end{remarks}
The next result establishes the existence of splines of complex order.
\begin{proposition}
Let $z \in\C_1$. The function $B_z :\R\to \C$ given by
\be\label{Bz}
B_z (x) = \frac{1}{\Gamma(z)} \sum_{k=0}^\infty (-1)^k \left(z \atop k\right) (x-k)_+^{z-1}
\ee
is a solution of Equation (\ref{cs}). 
\end{proposition}
\begin{proof}
First note that $B_z\in \Psi_+$. The linearity of $\mcD^z$ and Equation (\ref{dertruncpow}) imply that
\[
\mcD^z B_z (x) = \sum_{k=0}^\infty (-1)^k \left(z \atop k\right) \mcD^z \left[\frac{(x-k)_+^{z-1}}{\Gamma (z)}\right]
= \sum_{k=0}^\infty (-1)^k \left(z \atop k\right)\,\delta (x - k).
\]
We note that the infinite series $\displaystyle{\sum_{k=0}^\infty}\, (-1)^k \left(z \atop k\right)$ is bounded above by $c\,e^{|z-1|}$, $c > 0$. (See \cite{forster06} for a verification of this statement.)\qed
\end{proof}

The function $B_z$ defined in (\ref{Bz}) is called a {\em complex B-spline}. Complex B-splines were first introduced in \cite{forster06}  and later explored and generalized in \cite{FM07,FM09,FM09a,M09,FM10}. They provide a generalization of the classical Schoenberg polynomial splines to complex orders $z\in \C_1$ and contain the fractional B-splines defined in \cite{unserblu00}. 

Complex B-splines $B_z: \R \to\C$ were originally defined in the Fourier domain via the formula
\begin{equation}
\mathcal{F}(B_z)(\omega ) =: \widehat{B}_z (\omega)  = \left( \frac{1-e^{-i\omega}}{i\omega}\right)^z, \quad z\in \C_1.
\label{eq Definition Fourier B-Spline}
\end{equation}
Note that (\ref{eq Definition Fourier B-Spline}) possesses the continuous continuation $\widehat{B}_{z}(0)=1$ at the origin. As
\[
\left\{\frac{1-e^{-i\omega}}{i\omega}\;\Big\vert\; \omega\in \R\right\} \cap \left\{y \in\R \mid y<0\right\} = \emptyset,
\] 
complex B-splines reside on the main branch of the complex logarithm and are thus well-defined. For $z := n\in\N$, we obtain Schoenberg's polynomial cardinal splines.

\begin{remark}
For real $z > 0$, the function $\displaystyle{\Omega(z) := \left( \frac{1-e^{-i\omega}}{i\omega}\right)^z}$ and its time domain representation were already investigated in \cite{westphal} in connection with fractional powers of operators and later also in \cite{unserblu00} in the context of extending Schoenberg's polynomial splines to real orders. In the former, a proof that this function is in $L^1 (0,\infty)$ was given using arguments from summability theory (cf. Lemma 2 in \cite{westphal}), and in the latter the same result was shown but with a different proof. In addition, it was proved in \cite{unserblu00} that for real $z > 0$, $\Omega(z)\in L^2 (\R)$ for $z > 1/2$ (using our notation). (Cf. Theorem 3.2 in \cite{unserblu00}.)
\end{remark}

Complex B-splines possess several interesting basic properties, which are discussed in \cite{forster06}. In the following, we summarize the most important ones for our purposes. 

Fourier inversion of  (\ref{eq Definition Fourier B-Spline}) shows that complex B-splines are piecewise polynomials of complex degree. More precisely, the following result holds. (See \cite{forster06} for the proof.)

\begin{proposition}
 Complex B-splines have a time-domain representation of the form
\be\label{cb}
  B_z(t) = \frac{1}{\Gamma(z)} \sum_{k= 0}^\infty (-1)^k \left( {z} \atop {k}\right) (t-k)_+^{z-1},
\ee
where the above sum exists pointwise for all $t\in\R$ and in $L^2(\R)$-norm. \end{proposition}

Equation (\ref{cb}) shows that $B_z$ has, in general, non-compact support contained in $[0,\infty)$. It was also shown in \cite{forster06} that complex B-splines are elements of $L^1(\R) \cap L^2(\R)$ and, due to their decay in frequency domain induced by the polynomial $\omega^z$ in the denominator of (\ref{eq Definition Fourier B-Spline}), belong to the Sobolev spaces $W_2^r(\R)$ (with respect to the $L^2$-Norm and with weight $(1+|x|^2)^r$) for $r < \Re z - \frac{1}{2}$. The smoothness of their Fourier transform yields a fast decay in time domain: 
\be\label{asymptotics}
B_z (x) \in \mathcal{O}(x^{-m}), \quad \mbox{for $\N\ni m < \Re z +1$,  as $x \to \infty$}. 
\ee
\begin{remark}
Prior to \cite{forster06}, the asymptotic behavior (\ref{asymptotics}) of the function $\Omega(z)$ for real $z > 1$ was already shown in \cite{butzerwestphal}, (Proposition 3.1), to be of order $\mathcal{O}(x^{-z-1})$ as $x\to \infty$. The same estimate was proven later in \cite{unserblu00}, (Theorem 3.1), for real $z > 0$. As we are more interested in the approximation-theoretic aspects of complex B-splines, we restrict our attention to the case $\Re z > 1$, which yields continuous functions.
\end{remark}

If $z, z_1, z_2 \in \C_1$, then the convolution relation $B_{z_1} \ast B_{z_2} = B_{z_1+z_2}$ and the recursion relation 
\[
B_z (x) = \frac{x}{z-1}\,B_{z-1}(x) + \frac{z-x}{z-1}\,B_{z-1} (x - 1)
\]
hold.
Complex B-splines are scaling functions and generate multiresolution analyses of $L^2(\R)$ and wavelets. Furthermore, they relate difference and differential operators. For more details and proofs, we refer the interested reader to \cite{forster06,FM07,FM09,FM09a,FM11a}.

Unlike the classical cardinal B-splines, complex B-splines $B_z$ possess an additional modulation and phase factor in the frequency domain:
\[
\widehat{B}_z (\omega) = \widehat{B}_{\Re z}(\omega)\,e^{i \Im z \ln |\Omega(\omega)|}\,e^{- \Im z \arg \Omega(\omega)}.
\]
The existence of these two factors allows the extraction of additional information from sampled data and the manipulation of images. Phase information ($e^{i \Im z \ln |\Omega(\omega)|}$) and an adjustable smoothness parameter, namely $\Re z$, are already built into their definition. Thus, they define a {\em continuous} family, with respect to smoothness, of approximation spaces, and allow to incorporate phase information for single band frequency analysis. (See for instance \cite{forster06,FM11a}.)

In \cite{FM09} and \cite{FM10}, some further properties of complex B-splines were investigated. In particular, connections between complex derivatives of Riemann-Liouville or Weyl type and Dirichlet averages were exhibited. Whereas in \cite{FM09} the emphasis was on univariate complex B-splines and their applications to statistical processes, multivariate complex B-splines were defined in \cite{FM10} using a well-known geometric formula for classical multivariate B-splines \cite{KMR,micchelli}. It was also shown that Dirichlet averages are especially well-suited to explore the properties of multivariate complex B-splines. Using Dirichlet averages, several classical multivariate B-spline identities were generalized to the complex setting. The existence of fundamental complex B-splines was investigated in \cite{FGMS} and in \cite{FMU} periodic complex B-splines were constructed. There also exist interesting relationships between complex B-splines, Dirichlet averages and difference operators, several of which are highlighted in \cite{FM09a}. Here, we present some of these results to elucidate the connections between fractional operators, Dirichlet averages and splines of complex order. 

To this end, let $\diffm g(t) = g(t)- g(t-1)$ denote the backward difference operator for functions $g: \R \to \R$. The $n$th backward difference operator is iteratively defined by setting $\diffm^{n+1} g(t) := \diffm (\diffm^{n} g(t))$, $ n\in\N$. Then
$$
\diffm^n g(t) = \sum_{k=0}^n \left( n\atop k\right) (-1)^k g(t-k)
$$
vanishes for any function $g$ which is a polynomial of degree $\leq n-1$, $n\in\N$. These difference operators are uniquely related to the classical cardinal B-splines:
\begin{equation}
B_n(x) = \frac{1}{(n-1)!} \nabla^n_- x_+^{n-1} = \frac{1}{(n-1)!} \sum_{k=0}^n (-1)^k  \left( n \atop k \right) (x-k)_+^{n-1}.
\label{eq B-Spline}
\end{equation}
In \cite{unserblu00} and \cite{forster06}, it was shown that the backward difference operators have a natural formal extension to fractional and complex order:
$$
\diffm^z g(t) = \sum_{k \geq 0} (-1)^k \left( z \atop k\right) g(t-k), \quad z\in\C_+.
$$
Note that the right hand side is finite for bounded functions $g$ because $\sum_{k\geq 0} |({z \atop k})| \leq \const\, e^{|z-1|}$. (Cf. \cite{forster06}.)

The $n$th divided difference for a knot sequence $t_0 <t_1 <\ldots <t_n$ on the real line for functions $g: \R \to \R$ or $\C$ is recursively defined as
\begin{eqnarray}
[t_0]g & = & g(t_0) \quad \mbox{for } n=0,
\nonumber\\
{[ t_0,\ldots, t_n ]g} &=& \frac{[t_0,\ldots,t_{n-1}]g - [t_1,\ldots, t_{n}] g}{t_0 -t_n}
\quad \mbox{for } n \geq 1.
\label{eq Dividierte Differenzen}
\end{eqnarray}
The formulas have a nonrecursive equivalent which reads
$$
[t_0,\ldots, t_n]g = \sum_{j=0}^n \frac{g(t_j)}{\prod_{l\neq j} (t_j-t_l)}.
$$

The classical cardinal B-spline has a representation in terms of the $n$th divided difference on a uniform knot sequence $\{0,1,\ldots, n\}$ of the form
\begin{eqnarray}
B_n(x) & = & (-1)^n n [0,1,\ldots,n] (x - \bullet)_+^{n-1}.
\label{eq B-spline als dividierte Differenz}
\end{eqnarray}
For complex B-splines, a similar computation produces
\begin{eqnarray}
B_z(x) 
&= & z \sum_{k\geq 0} (-1)^k \frac{1}{\Gamma(z-k+1) \Gamma(k+1)} (x-k)_+^{z-1}.
\label{eq z mal Form von komplexen B-splines}
\end{eqnarray}

The above representation of complex B-splines motivates the definition of a complex divided difference operator $[z;\N_0](\bullet)$ of order $z\in\C_+$ for uniform knots $\{0,1,2,\ldots\}  = \N_0$:
\begin{equation}
[z;\N_0]g  := \sum_{k\geq 0} (-1)^k \frac{g(k)}{\Gamma(z-k+1) \Gamma(k+1)}.
\label{eq Complexe Dividierte Differenz}
\end{equation}
Then
\begin{equation}
B_z(x) = z [z; \N_0](x - \bullet)_+^{z-1}.
\label{eq Complex B-spline als dividierte Differenz}
\end{equation}
For $z= n\in\N_0$, Eqns. (\ref{eq z mal Form von komplexen B-splines}), (\ref{eq Complexe Dividierte Differenz}), and
(up to a factor $(-1)^n$) (\ref{eq Complex B-spline als dividierte Differenz}) reduce to the standard forms (\ref{eq B-Spline}), (\ref{eq Dividierte Differenzen}), and (\ref{eq B-spline als dividierte Differenz}) for $n$th order finite differences on uniform knots and classical B-splines.

We recall the following relation between the $n$-th order cardinal B-spline $B_n$, $n\in\N$, and the divided differences:
\begin{equation}
[0,1,\ldots,n] g = \frac{1}{n!} \int_\R B_n(t) D^n g(t) \, dt.
\label{eq Verbindung Dividierte Differenzen und B-splines als Integrationskern}
\end{equation}
An analogue for complex B-splines reads as follows. (See \cite{FM09}.)

\begin{theorem}  Suppose that $z \in \C_1$ and $g\in \Psi_+$. Then the complex B-spline $B_z$ and the complex divided difference (\ref{eq Complexe Dividierte Differenz}) obey the following relation.
\be\label{interpretation}
[z;\N_0] g = \frac{1}{\Gamma(z)} \int_\R B_z(t) \mcD^{z}g (t)\, dt.
\ee
\end{theorem}

The relationship (\ref{interpretation}) allows an interesting interpretation of a complex B-spline. For this purpose, we introduce the complex forward difference operator $\diffp^z$, $z \in \C_+$, acting on functions $g: \R\to\R$ via
\[
(\diffp^z g)(t):= \sum_{k\geq 0} (-1)^k \left(z \atop k\right) g(t+k).
\]
Then, if we interpret the integral $\int_\R B_z (t) g (t) dt$ as a weak integral as in Definition 2.9 in \cite{fuehr}, namely, 
\[
\int_\R B_z (t) g (t) dt =: \left\langle \int_\R B_z (t)\; \bullet\;dt, g\right\rangle,
\]
then we obtain
\begin{eqnarray*}
\left\langle \int_\R B_z(t)\; \bullet\; dt, g\right\rangle & = & \displaystyle{\int_\R B_z (t) g (t) dt}\\
& = & \displaystyle{\int_\R\sum_{k\geq 0} (-1)^k \left(z \atop k \right) \frac{1}{\Gamma (z)}\, (t - k)_+^{z - 1} g(t) dt}\\
& = & \displaystyle{\sum_{k\geq 0} (-1)^k \left( z \atop k \right)(\mcD^{-z} g)(k) = (\diffp^z \mcD^{-z} g)(0) = \left\langle\delta,\diffp^z \mcD^{-z} g \right\rangle}.
\end{eqnarray*}
In particular, substituting $\mcD^z g$ for $g$, we immediately obtain the following identity:
\[
\left\langle \int_\R B_z(t)\; \bullet\; dt, \mcD^{z}g\right\rangle = \left\langle\delta,\diffp^z g \right\rangle.
\]
Hence, $\left\langle \int_\R B_z(t)\; \bullet\; dt, \,\bullet\,\right\rangle$ can be thought of as a fractional integration operator of order $z$.

Next, we like to present an interesting and deep connection between complex B-splines, Dirichlet averages and divided differences. This connection is of stochastic nature and the case of integral order $n$ can be found in \cite{dm} and its extension to complex orders in \cite{FM09}.

To this end, $\tau := \{\tau^j \st j\in\mathbb{N}_0\}$ be an infinite increasing sequence of positive real-valued knots such that condition (\ref{cond}) applied to this particular setting holds. Furthermore, define a random variable $X := \sum_{j=0}^\infty \tau^j U_j$, where $U := (U_j)$ is an infinite random vector uniformly distributed over the simplex $\Delta^\infty$. We denote the expectation of a random variable by $\mathsf{E}$. Then is was shown in \cite{FM09} that
\be\label{111}
\mathsf{E} g(X) = \int_{\Delta^\infty} g(\tau\cdot u)\,d\mu_b (u),\quad \forall\,g\in C(\R),
\ee
and, using the fact that $\mcD^z$ is an endomorphism on $\Psi_+$, that
\be\label{Egz}
\mathsf{E} (\mcD^{z} g)(X) = \int_{\Delta^\infty} \mcD^z g(\tau\cdot u)\,d\mu_b (u).
\ee
\par
Now, let us consider the case $b =  e = (1,\ldots, 1)\in\mathbb{R}^{\infty}_+$. Then the following analog of the Hermite-Genocchi for complex B-splines was proved in \cite{FM09}.

\begin{theorem}[Hermite--Genocchi]\label{HG}
Let $z\in\mathbb{C}_1$ and let $B_z$ be the complex B-spline of order $z$. Then
\begin{eqnarray*}
[z;\mathbb{N}_0]g & = & \displaystyle{\frac{1}{\Gamma(z+1)}\,\int_{\Delta^\infty} \mcD^z g (\mathbb{N}_0\cdot u) d\mu_e (u) = \frac{1}{\Gamma(z+1)}\,\mathsf{E} (\mcD^z g) (X)}\\
& = & \displaystyle{\frac{1}{\Gamma(z+1)}\,\int_{\mathbb{R}} B_z (t) \mcD^z g(t) dt =(\diffp^z g)(0)} ,
\end{eqnarray*}
for all $g\in\Psi_+\cap C^\omega (\R)$, where $C^\omega (\R)$ denote the class of real-analytic functions on $\R$.
\end{theorem}

Theorem \ref{HG} can be used to defined a larger class of B-splines of complex order. In particular, we like to consider an arbitrary increasing knot sequence and attach weights to each knot.
\begin{definition}[Cf. \cite{FM09}]
Let $b\in(\C^+)^{\N_0}$ be a weight vector and $\tau:= \{\tau^k\st \tau^0 = 0\} \in \R^{\N_0}$ an increasing knot sequence with the property that $\lim_{k\to \infty} \sqrt[k]{\tau^k} \leq \varrho$, for some $\varrho\in [0,e)$. Furthermore, let  $z\in\C_1$. A function $B_z(\bullet \mid b; \tau)$  satisfying
\begin{equation}
\int_\R B_z(t \mid b;\tau) \mcD^z g(t)\, dt = \int_{\triangle^\infty} \mcD^z g(\tau \cdot u)\, d \mu_b(u) 
\label{eq Def complex B-spline}
\end{equation}
for all $g\in \Psi_+$ is called a B-spline of complex order $z$ with weight vector $b$ and knot sequence $\tau$.
\end{definition}

We remark that for finite knot vector $\tau = \{\tau_0, \tau_1, \ldots, \tau_n\}\in(\R_0^+)^{n+1}$ and finite weight vector $b = \{b_0, b_1, \ldots, b_n\}\in (\R^+)^n$, $n\in \N$, and $z:= n\in\N$, Eq.~(\ref{eq Def complex B-spline}) defines the so-called {\em Dirichlet splines}. (Cf. \cite{dm}, where these splines were first introduced.) 

\vspace*{10pt}
To extend B-splines of complex order to a multivariate setting, one employs ridge functions. To this end, let $\lambda \in \R^s\setminus \{0\}$, $s\in \N$, be a direction, and let $g: \R \to \C$ be a function. The {\em ridge function}  $g_\lambda$ corresponding to $g$ in the direction of $\lambda$ is defined as the function $\R^s \to \C$ with
$$
g_\lambda(x) := g(\inn{\lambda}{x}), \quad\textrm{for all $x\in\R^s$}.
$$

\begin{definition} [See also \cite{FM10}]
Let $z\in\C_1$ and let $\tau=\{\tau^n\st \tau^0 = 0\}\in (\R^s)^{\N_0}$ be a knot sequence in $\R^s$ with the property that
\[
\exists\,\varrho\in [0,e):\,\displaystyle{\limsup_{n\to\infty}} \sqrt[n]{\|\tau^n\|} \leq \varrho. 
\]
Furthermore, let $\lambda\in\R^s\setminus\{0\}$ be such that $\lambda\tau := \{\skpl \lambda, \tau^n\skpr\}_{n\in\N_0}$ is separated, i.e., there exists a $\delta > 0$, so that $\inf \{|\skpl \lambda, \tau^n\skpr - \skpl \lambda, {\tau}^m\skpr|\st m, n\in \N_0\} \geq \delta$.

A function $\bB_z(\bullet \mid b,\tau): \R^s \to \C$ which satisfies the equation
\be\label{multivariate complex}
\int_{\R^s} g(\inn{\lambda}{x}) \bB_z(x\mid b,\tau)\, dx = \int_\R g(t) B_z(t\mid b, \lambda\tau) \, dt,
\ee
for all $g\in\Psi_+$ is called a multivariate B-spline of complex order $z$ with weight vector $b\in(\C^+)^{\N_0}$ and knot sequence $\tau$.
\end{definition}

Since ridge functions are dense in $L^2(\R^s)$ (see, for instance, \cite{pinkus97}), we conclude that $\bB_z(\bullet \mid b, \,{\tau}) \in L^2((\R^+_0)^s)$. Moreover, it follows from the Hermite-Genocchi formula for the univariate complex B-splines $B_z(\,\bullet\,|\, b, \lambda\tau)$ and (\ref{multivariate complex}), that 
$$\bB_z(\,x\,|\, b,\tau) = 0,\quad\textrm{when $x\notin [\tau]$},
$$ 
where $[\tau]$ denotes the convex hull of $\tau$.

Considering the special case $b=e= (1,1,1, \ldots)$, we may define multivariate divided differences of order $z$ on ridge functions as follows:
\begin{eqnarray*}
[z;\tau]g_\lambda & = & [z;\tau]g(\skpl \lambda, \bullet\skpr) =  \frac{1}{\Gamma(z)} \int_{\R^s} g^{(z)}(\skpl \lambda, x\skpr) \bB_z(x\mid e,\tau)\, dx\\
& =  & \frac{1}{\Gamma(z)} \int_\R g^{(z)}(t) B_z(t\mid e, \lambda \tau) \, dt =  [z; \lambda\tau]g,\qquad \forall g\in \Psi_+^\infty. 
\end{eqnarray*}
We refer the interested reader to \cite{FM11a} for more details and further results.

Next, we like to investigate the Fourier representation of multivariate B-splines of complex order and obtain a time domain representation for $\bB_z(x\mid b,\tau)$ as well. 

For this purpose recall that $B_{z}(\bullet\mid b;\tau)$ and $\bB_z (\bullet \mid b;\boldsymbol{\tau})$ are elements of $L^2(\R^s)$. Hence, we can apply the Plancherel transform to both functions and consider their frequency spectrum.

Let $\boldsymbol{\omega} = (\omega_1,\ldots, \omega_s) \in \R^s$ and let $\boldsymbol{\lambda} \in \R^s$, $\Vert \boldsymbol{\lambda} \Vert =1$, be the direction of $\boldsymbol{\omega} $, i.e., $\boldsymbol{\omega} = \varpi \boldsymbol{\lambda}$ for some $\varpi  \geq 0$. For $\boldsymbol{x} = (x_1,\ldots, x_s)\in\R^s$, we obtain as the Fourier transform of the generalized complex B-spline the following expression:
\begin{eqnarray*}
\widehat{B}_z(\varpi \mid b; \boldsymbol{\lambda}\boldsymbol{\tau}) & = & 
\int_{\R} e^{-i\varpi t} B_z(t\mid b; \boldsymbol{\lambda}\boldsymbol{\tau})\, dt
= \int_{\R^s} e^{-i\varpi \langle \boldsymbol{\lambda}, \boldsymbol{x} \rangle} \bB_z(\boldsymbol{x}\mid b;\boldsymbol{\tau}) \, d\boldsymbol{x}
\\
& = & \int_{\R^s} e^{-i \langle \boldsymbol{\omega}, \boldsymbol{x} \rangle}\bB_z(\boldsymbol{x}\mid b;\boldsymbol{\tau}) \, d\boldsymbol{x}
= \widehat{\bB}_z(\boldsymbol{\omega} \mid b; \boldsymbol{\tau}) = 
\widehat{\bB}_z(\varpi\boldsymbol{\lambda} \mid b; \boldsymbol{\tau}).
\end{eqnarray*}
Hence, the frequency spectrum of the multivariate complex B-spline along directions $\boldsymbol{\lambda}$ is given by the spectrum of the univariate spline with knots projected onto these $\boldsymbol{\lambda}$.

We now consider the special case when the knots are equidistantly distributed on a ray in $\R^s$, i.e., $\boldsymbol{\tau} = \boldsymbol{d} \N_{0}$ for some distance vector $\boldsymbol{d}\in \R^s$. Then we obtain
\begin{eqnarray*}
\widehat{\bB}_z(\varpi\boldsymbol{\lambda} \mid b; \boldsymbol{\tau}) & = & \widehat{\bB}_{z} (\varpi \boldsymbol{\lambda} \mid b; \boldsymbol{d} \N_{0}) = \widehat{\bB}_{z} (\varpi  \mid b; \skpl \boldsymbol{\lambda},\boldsymbol{d}\skpr \N_{0}),
\end{eqnarray*}
and specifically for weights $b=e= (1,1,1,\ldots)$:
\be\label{Standard multi Spline in Fourier domain}
\widehat{\bB}_z(\varpi\boldsymbol{\lambda} \mid e; \boldsymbol{d} \N_{0}) = 
\left( 
\frac{1- e^{-i \skpl \boldsymbol{\lambda},\boldsymbol{d}\skpr \varpi}}{i \varpi}
\right)^z, \qquad \textrm{in $\;L^2(\R^s$.}
\label{eq Standard multi Spline in Fourier domain}
\ee
Note that
$$
\widehat{\bB}_z(\varpi\boldsymbol{\lambda} \mid e; \boldsymbol{d} \N_{0}) = 0 \quad \gdw \quad \varpi \skpl \boldsymbol{\lambda},\boldsymbol{d}\skpr = 2 \pi K, \quad K\in\Z.
$$

Finally, we state an explicit formula for the time domain representation of a multivariate B-spline of complex order in case the weights are $b=e= (1,1,1,\ldots)$.

\begin{theorem}
The time domain representation of a multivariate B-spline of complex order with weight vector $b=e= (1,1,1,\ldots)$ is given by
\[
\bB_{z}(\boldsymbol{x} \mid e, \boldsymbol{d}\N_{0}) = \frac{\pi^{\frac{z}{2}-s}}{i^z \, 2^{\frac{z+3s}{2}}} \frac{\Gamma\left(\frac{s-z}{2}\right)}{\Gamma\left(\frac{z}{2}\right)}\sum_{n=0}^\infty (-1)^n \left( z \atop k \right) 
\Norm  \boldsymbol{x} - n \boldsymbol{d} \Norm^{z-s},
\]
in the sense of tempered distributions in $\Psi_+$ and in $L^2(\R^s)$.
\end{theorem}

\begin{proof}
The proof employs the inverse Fourier transform of (\ref{Standard multi Spline in Fourier domain}). For more details, please consult \cite{FM11a}.
\end{proof}

So far we have considered polynomial-type splines, i.e., solutions of Equation (\ref{cs}). It is a natural question to ask whether there are other types of splines that are the solution of more general differential operators. The answer is affirmative and found in \cite{karlin} for orders $n\in\N$. In \cite{unserblu05} exponential splines were defined as the solution of linear differential operators $\mcL$ with constant coefficients. (See also \cite{masso}.) 

Recently, a univariate {\em exponential spline $E^a_z$ of complex order} $z\in\C_1$ was introduced in \cite{M13}. This particular spline is a solution of the (distributional) differential equation
\[
(\mcD + a I)^{z} f = \sum_{\ell=0}^\infty c_\ell \,\delta (\bullet -\ell),\quad a\in \R,
\]
where $I$ denotes the identity operator on $\Psi_+$ and $\{c_\ell\st \ell\in \N\}$ is an $\ell^\infty$-sequence, and possesses a Fourier domain representation of the form
\[
\widehat{E}_z^a (\omega,a):=\left( \frac{1-e^{-(a+i\omega)}}{a+i\omega}\right)^z.
\]
The function $\Omega(\omega, a) := \left( \frac{1-e^{-(a+i\omega)}}{a+i\omega}\right)^z $ is only well-defined for $a \geq 0$. Results for and properties of this new class of splines will be presented elsewhere.

\input{references.tex}

\end{document}

%% file: references.tex
%
%
%